\documentclass[preprint,12pt]{elsarticle}

\usepackage{amsmath}
\usepackage{amssymb}
\usepackage{amsthm}
\usepackage{graphicx}
\usepackage{hyperref}
\usepackage{physics}
\usepackage{subfig}
\usepackage{float}
\usepackage{pdflscape}
\usepackage{mathtools}
\usepackage{algorithm} %
\usepackage{algpseudocode} %

\newcommand{\E}{\mathbb{E}}

\newcommand{\expvalb}[2]{\E_{#1}\left[ #2 \right]}
\newcommand{\inner}[1]{\left\langle #1 \right\rangle}

\newtheorem{thm}{Theorem}
\newtheorem{prop}{Proposition}
\newtheorem{lem}{Lemma}

\journal{-}

\begin{document}

\begin{frontmatter}

\title{A Bayesian Update Method for Exponential Family Projection Filters with Non-Conjugate Likelihoods}

\author[firstaddress,secondaddress]{Muhammad Fuady Emzir\corref{mycorrespondingauthor}}
\cortext[mycorrespondingauthor]{Corresponding author}
\ead{muhammad.emzir@kfupm.edu.sa}
\address[firstaddress]{Control and Instrumentation Engineering Department, King Fahd University of Petroleum and Minerals, Dhahran, Saudi Arabia}
\address[secondaddress]{Interdisciplinary Research Center of Smart Mobility and Logistics, King Fahd University of Petroleum and Minerals, Dhahran, Saudi Arabia}

\begin{abstract}
  The projection filter is one of the approximations to the solution of the optimal filtering problem. It approximates the filtering density by projecting the dynamics of the square-root filtering density onto the tangent space of the square-root parametric density manifold. While the projection filters for exponential and mixture families with continuous measurement processes have been well studied, the continuous-discrete projection filtering algorithm for non-conjugate priors has received less attention.
  In this paper, we introduce a simple Riemannian optimization method to be used for the Bayesian update step in the continuous-discrete projection filter for exponential families. Specifically, we show that the Bayesian update can be formulated as an optimization problem of $\alpha$-R\'enyi divergence, where the corresponding Riemannian gradient can be easily computed. We demonstrate the effectiveness of the proposed method via two highly non-Gaussian Bayesian update problems.
  
\end{abstract}


\end{frontmatter}

\section{Introduction}

The projection filter is one of the approximations for nonlinear filtering solutions \cite{hanzon1991,brigo1998,brigo1999}. The projection filter can be seen as a rigorous treatment of the assumed density filtering developed in the 60s \cite{kushner1967}. Recent developments in the projection filter via sparse-grid integration enable efficient implementation of the projection filter for multi-dimensional problems and have renewed interest in this topic; see \cite{emzir2023,emzir2023a,emzir2024c,emzir2024}. Unlike the case for continuous measurement processes, the continuous-discrete projection filter has not been well-studied. In the case of the exponential family manifold, constructing a projection filter algorithm for a discrete measurement process is challenging when the likelihood function is not conjugate to the chosen exponential family. This difficulty arises from the need to project the square root of the posterior density onto the square-root exponential family manifold. Unlike the scenario with a conjugate prior, this projection is non-trivial. In this paper, we propose a variational method to address the Bayesian update process in such cases.

To the best of the author's knowledge, the application of a variational method to address the Bayesian update step for the exponential family's projection filter with non-conjugate priors has not been previously proposed in the literature. However, equivalent applications of the variational method to nonlinear filtering problems exist. For instance, in \cite{darling2017a}, the Kullback--Leibler (KL) divergence is minimized via moment matching \cite{herbrich2005} in the Bayesian update step for the assumed Gaussian density filter; see also \cite{kulhavy1990,smidl2008,corenflos2023}. We argue that Riemannian optimization is more suitable for finding the closest parametric density to the actual posterior since the set of parametric densities can be regarded as a Riemannian manifold. It has been observed that Riemannian gradient descent offers better convergence compared to non-Riemannian gradient-based optimization methods in various information geometric problems; see \cite{amari1998}. Moreover, instead of employing KL divergence as the cost function, as is common practice in many variational-based filtering approaches, we opt for $\alpha$-R\'enyi divergence, which serves as a generalization of KL divergence \cite{li2016,saha2020}.

The contributions of this paper are twofold. First, we formulate the Bayesian update procedure for the continuous-discrete projection filter as a Riemannian optimization problem. We derive the Riemannian gradient expression of the $\alpha$-R\'enyi divergence function, which can be used to identify the closest parametric density to the posterior. The optimization method was not needed in continuous-continuous filtering problems, since there is no explicit Bayesian update, and the projection filter is applied directly to the Stratonovich--Kushner stochastic partial differential equation. The optimization problem is also not needed in the continuous-discrete case with a conjugate prior, as the update step is exact. We further show that under the Riemannian gradient descent parameter update, the set of points in the parameter space with vanishing gradients is globally asymptotically stable. Second, we show how to implement Riemannian gradient descent via sparse-grid quadrature using an adaptive bijection function. Specifically, we highlight the advantages of minimizing $\frac{1}{2}$-R\'enyi divergence over KL divergence in approximating two highly non-Gaussian posteriors in Bayesian update steps through two numerical examples. The proposed Bayesian update can then be implemented for the continuous-discrete projection filter on the exponential-family manifold, providing an efficient method to approximate the solution to optimal filtering problems.

The paper is organized as follows. Section \ref{sec:Notation} provides all necessary notations used in this paper. Section \ref{sec:Projection_Filter} introduces the continuous-discrete projection filter formulation for the exponential family manifold. Section \ref{sec:Bayesian_update} contains the main contributions of this paper, where the proposed variational method for the Bayesian update phase is introduced. Section \ref{sec:numerical_implementation} details the numerical implementation of the proposed Bayesian update. Section \ref{sec:Numerical_Examples} presents two numerical examples that highlight the effectiveness of the proposed method. Finally, Section \ref{sec:Conclusions} summarizes the findings of this paper.

\section{Notation}\label{sec:Notation}
For an $m$-dimensional manifold $M$ with a chart $(U, \phi)$, we denote $\partial_i \coloneqq \pdv{}{\phi^i}$; i.e., for a point $p \in U$, a germ $f \in C_p^\infty(U)$, and $r^i$ as the $i$-th coordinate of $\mathbb{R}^m$, $\left.\pdv{ }{\phi^i}\right|_p f =\left.\pdv{}{r^i}\right|_{\phi(p)} \left(f \circ \phi^{-1} \right)$. For the Fisher information matrix $g(\theta)$, we place its two indices down; i.e., $g(\theta)_{ij}$, while for its inverse, the indices are shown up; i.e., $g(\theta)^{ij}$. Further, we denote the tangent space of a manifold $M$ at a point $p$ as $T_p M$ and the tangent bundle of $M$ as $TM$. For a smooth mapping $F$ between two manifolds $M$ and $N$, we denote $F_\ast$ as the differential of $F$, such that for $X_p \in T_pM$, $F_\ast X_p \in T_{F(p)} N$ is the \emph{push-forward} of $X_p$. For a parametric density $p_\theta$, we denote $\E_\theta \left[ \cdot \right] := \E_{p_\theta} \left[ \cdot \right]$.

\section{Continuous-Discrete Projection Filter for the Exponential Family}\label{sec:Projection_Filter}
In this section, we review the relevant theoretical results that constitute the foundation of the projection filter for the exponential family \cite{brigo1995,brigo1999}. The aim of the projection filter is to approximate the evolution of the filtering density (which is generally infinite-dimensional) with a finite-dimensional evolution of the natural parameters corresponding to the exponential family. Consider optimal filtering problems on the following state-space model consisting of continuous-time stochastic dynamics and discrete observation models:
\begin{subequations}
  \label{eqs:nonlinear_SDE}
  \begin{align}
    dx_t &= f(x_t) \, dt + \varrho(x_t) \, dW_t,\label{eq:state_nonlinear_SDE}\\
    y_k &\sim p(y_k\mid x_k) \propto \exp(-\ell(x_k,y_k)),\label{eq:likelihood}
  \end{align}
\end{subequations}
where, for a positive sampling interval $\Delta t$, $y_k\coloneqq y_{k\Delta t}$, and $x_t \in \mathbb{R}^d$, $y_k \in \mathbb{R}^{d_y}$. $\{W_t, t \geq 0\}$ is a Wiener process taking values in $\mathbb{R}^{d_w}$, and $\ell(x_k,y_k)$ is the negative log-likelihood function of the discrete measurement process. The evolution of the probability density corresponding to the SDE \eqref{eq:state_nonlinear_SDE} is governed by the Fokker--Planck equation
\begin{equation}
  \pdv{p_t}{t} = \mathcal{L}^\ast(p_t), \label{eq:Fokker_Planck_Equation}
\end{equation}
where $\mathcal{L}^\ast$ is the adjoint of the Kolmogorov operator and is defined as
\begin{equation}
  \mathcal{L}^\ast(p_t) = -\sum_{i=1}^d \pdv{}{x_i} \left( f_i(x) p_t \right) + \frac{1}{2} \sum_{i,j=1}^d \pdv[2]{}{x_i}{x_j} \left( \varrho_{ij}(x) p_t \right). \label{eq:Kolmogorov_Operator}
\end{equation}

Let us define a class of probability densities $\mathcal{P}$ with respect to the Lebesgue measure on a fixed domain $\mathcal{X} \subseteq \mathbb{R}^d$ as
\begin{equation}
  \mathcal{P} = \{p \in L^1 : \int_\mathcal{X} p(x) \, dx = 1,\ p(x) \geq 0,\ \forall x \in \mathcal{X}\}.
\end{equation}
The exponential family is defined as
\begin{align}
    \mathrm{EM}(c) \coloneqq \left\{ p \in \mathcal{P} \colon p(x) = \exp(c(x)^\top\theta - \psi(\theta)) \right\},\label{eq:EM_c}
\end{align}
where $\theta \in \Theta \subset \mathbb{R}^m$ is the natural parameter, and $c\colon\mathbb{R}^d \to \mathbb{R}^m$ is a vector of natural statistics that are assumed to be linearly independent. The natural parameter space $\Theta$ is defined as the collection of all $\theta$ such that the corresponding density $p_\theta$ is in $\mathcal{P}$, i.e.,
\begin{align}
    \Theta \coloneqq \left\{ \theta \in \mathbb{R}^m\colon \int_\mathcal{D} \exp(c(x)^\top \theta) \, dx < \infty \right\}, \label{eq:Theta}
\end{align}
where $\mathcal{D} \subseteq \mathcal{X}$ is the support of $\exp(c(x)^\top \theta)$. An exponential family is said to be regular if $\Theta$ is an open subset of $\mathbb{R}^m$. The cumulant-generating function (i.e., the log Laplace transform or log partition function \cite{brown1986,emzir2023}) is defined by
\begin{align}
    \psi(\theta) = \log \left[ \int_\mathcal{D} \exp(c(x)^\top \theta) \, dx \right], \quad \theta \in \Theta.
    \label{eq:cumulant-generating}
\end{align}
Because the exponential family is assumed to be regular and the natural statistics are linearly independent, the exponential family is minimal \cite{kass1997,calin2014}. We recall the following standard result for a minimal regular exponential family.
\begin{thm}(Theorems 2.2.1 and 2.2.5 of \cite{kass1997})\label{thm:kass_exponential_family_expectation}
    In a regular exponential family, the set $\Theta$ as defined in \eqref{eq:Theta} is convex. The cumulant-generating function $\psi(\theta)$ is strictly convex on $\Theta$ and is differentiable up to an arbitrary order. The moments of the natural statistics $c_i(x)$, $i=1,\ldots,m$, exist for any order, and the expectations of
     $c_i$ and the corresponding Fisher information matrix $g$ are, respectively, given by:
    \begin{align}
        \E_\theta\left[ c_i \right] &= \pdv{\psi(\theta)}{\theta_i}, &
        g(\theta)_{ij} &= \pdv[2]{\psi(\theta)}{\theta_i}{\theta_j}. \label{eq:eta_g_pdv}
    \end{align}
    If the representation is minimal, then $g$ is positive definite.
\end{thm}

Following \cite{brigo1999}, the projection filter described here is developed on the manifold of square-root exponential densities 
\begin{equation*}
  \text{EM}(c)^{\frac{1}{2}} := \{\sqrt{p_\theta} : p_\theta \in \text{EM}(c)\}.
\end{equation*}
The use of the space of square-root densities $\text{EM}(c)^{\frac{1}{2}}$ is due to the fact that the Fisher information metric can be defined via the standard $L^2(\mathcal{X})$ inner product. As is common in information geometry \cite{amari2000}, we work with a single chart $(\text{EM}(c)^{\frac{1}{2}}, \phi)$, where 
\begin{equation*}
  \phi : \text{EM}(c)^{\frac{1}{2}} \to \Theta \subset \mathbb{R}^m,\quad
  \phi(\sqrt{p_\theta}) \coloneqq \theta,\quad
  \sqrt{p_\theta} \coloneqq \sqrt{p(\cdot;\theta)},
\end{equation*}
(see also \cite{brigo2017,brigo2009} for an alternative representation via the exponential statistical manifold developed in \cite{pistone1995}). The differential of the map $\phi$ is denoted by $\phi_\ast$; i.e., for $X \in T_{\sqrt{p_\theta}}\mathrm{EM}(c)^{\frac{1}{2}}$, $\phi_\ast X \in T_\theta \mathbb{R}^m \cong \mathbb{R}^m$. To reduce technicality, we assume that the support $\mathcal{D}$ of the parametric density $p_\theta \in \text{EM}(c)$ is uniform across $\text{EM}(c)$; see \cite{amari2000}. We equip the manifold of square-root parametric densities with the Fisher information metric, given by
\begin{equation}
  \inner{\partial_i,\partial_j}_{\sqrt{p_\theta}}
  = \int_\mathcal{D} \frac{\partial \sqrt{p_\theta(x)}}{\partial \theta_i}\,\frac{\partial \sqrt{p_\theta(x)}}{\partial \theta_j}\,dx
  = \frac{1}{4}\,g(\theta)_{ij}. \label{eq:metric}
\end{equation}
With this metric, the square-root parametric density manifold becomes a Riemannian manifold, where notions of inner product, distance, and projection are well-defined. In particular, the projection onto the tangent space $T_{\sqrt{p_\theta}}\mathrm{EM}(c)^{\frac{1}{2}}$, 
$\Pi_{\sqrt{p_\theta}}:L^2(\mathcal{X})\to T_{\sqrt{p_\theta}}\mathrm{EM}(c)^{\frac{1}{2}}$, is given by
\begin{equation}
  \Pi_{\sqrt{p_\theta}}(v)
  = \sum_{i=1}^m\sum_{j=1}^m 4\,g(\theta)^{ij}\,\inner{v,\partial_j}_{\sqrt{p_\theta}}\,\partial_i.
  \label{eq:projection_onto_T_EM}
\end{equation}

Using the projection $\Pi_{\sqrt{p_\theta}}$, the continuous-discrete projection filter is implemented in two steps. In the first step, between sampling times $(k-1)\Delta t$ and $k\Delta t$, the dynamics of the square-root filtering densities are projected onto the tangent space $T_{\sqrt{p_\theta}}\mathrm{EM}(c)^{\frac{1}{2}}$. Assuming that at time $t$, $\sqrt{p_t} = \sqrt{p_{\theta_t}}$ for some $\theta_t \in \Theta$, and $d\sqrt{p_{\theta_t}}\in L^2(\mathcal{X})$ (see \cite[Lemma 2.1]{brigo1999}), the evolution of the natural parameters is defined via $\Pi_{\sqrt{p_\theta}}(d\sqrt{p_{\theta_t}}) = \sum_{i=1}^m \dv{\theta^i}{t} \partial_i$, giving 
\begin{gather}
  \begin{aligned}
    \dv{\theta^i}{t}
    &= \sum_{j=1}^m 4\,g(\theta)^{ij}\,\inner{d\sqrt{p_t},\partial_j}_{\sqrt{p_{\theta_t}}}\\
    &= \sum_{j=1}^m 4\,g(\theta)^{ij}\,\inner{\tfrac{1}{2\sqrt{p_t}}\mathcal{L}^\ast(p_t),\partial_j}_{\sqrt{p_{\theta_t}}}\\
    &= \sum_{j=1}^m 4\,g(\theta)^{ij}\,\inner{\tfrac{1}{2\sqrt{p_{\theta_t}}}\mathcal{L}^\ast(p_{\theta_t}),\partial_j}_{\sqrt{p_{\theta_t}}}\\
    &= \sum_{j=1}^m g(\theta)^{ij}\,\E_{\theta}[\mathcal{L}(c_j)],
    \label{eq:projected_Fokker_Planck}
  \end{aligned}
\end{gather}
where $\mathcal{L}$ is the backward Kolmogorov diffusion operator: for a test function $\varphi$,
\begin{equation}
  \mathcal{L}(\varphi)
  = \sum_{i=1}^df_i(x)\,\pdv{\varphi(x)}{x_i}
  + \tfrac{1}{2}\sum_{i,j=1}^d\varrho_{ij}(x)\,\pdv[2]{\varphi(x)}{x_i}{x_j}.
  \label{eq:Kolmogorov_Operator_Backward}
\end{equation}
At the end of the propagation, we obtain $\theta_k^- := \theta_{k\Delta t}$, corresponding to the predictive density $p_{\theta_k^-} \in \mathrm{EM}(c)$.

The second step is the correction step, where the measurement $y_k$ is incorporated to update $p_{\theta_k^-}$ via the likelihood density $p(y_k\mid x_k) \propto \exp(-\ell(x_k,y_k))$. The posterior density $q$ is given by\footnote{Although the vector of natural statistics $c$ is deterministic, the expectation $\E_{\theta_k}[c]$ and other expectations with respect to $p_{\theta_k}$ (such as those in \eqref{eqs:posterior_and_Z} and \eqref{eq:projected_Fokker_Planck}) are stochastic processes, since $\theta_k$ is updated according to the measurement $y_k$.}
\begin{subequations}
  \begin{align}
    q &= p_{\theta_k^-}\,\exp(-\ell(\cdot,y_k)-Z(\theta_k^-,y_k)), \label{eq:posterior}\\
    Z(\theta_k^-,y_k)
      &= \log(\E_{\theta_k^-}[\exp(-\ell(\cdot,y_k))]).
  \end{align}\label{eqs:posterior_and_Z}
\end{subequations}
From this perspective, $q$ as defined in \eqref{eq:posterior} is the posterior density under Bayes' rule, given that the prior density is $p_{\theta_k^-}$. However, $q$ is not the true posterior density of the state $x_k$ given the measurements $y_{1:k}$, unless the true predictive density is equal to $p_{\theta_k^-}$, which is generally not the case. In this paper, since our focus is only on a single Bayesian update step, we refer to $q$ as the posterior density.

When the negative log-likelihood has the form $\ell(x_k,y_k)=(y_k-h(x_k))^\top(y_k-h(x_k))$, with $h$ and $h^\top h$ in the span of the natural statistics, the posterior density remains in $\mathrm{EM}(c)$ and the corresponding natural parameters can be calculated exactly; see \cite[\S 6.2]{brigo1999}. In the next section, we develop a method to generalize the Bayesian update to non-conjugate priors by finding $\sqrt{p_\theta}\in\mathrm{EM}(c)^{\frac{1}{2}}$ that minimizes a divergence function via Riemannian optimization.

\section{Bayesian Update Algorithm}\label{sec:Bayesian_update}
In this section, we derive a Riemannian gradient descent algorithm to minimize the dissimilarity between the actual posterior $q$ given in \eqref{eq:posterior} and the approximated posterior $p_{\theta_k} \in \text{EM}(c)$, as measured by divergences. To do so, we need to define a movement from a point $\sqrt{p_\theta}$ on the manifold $\text{EM}(c)^{\frac{1}{2}}$ in the direction of a vector $X \in T_{\sqrt{p_\theta}} \text{EM}(c)^{\frac{1}{2}}$ via a curve $\gamma(t)$ on $\text{EM}(c)^{\frac{1}{2}}$. The curve is needed since the addition $\sqrt{p_\theta} + \delta X$, for $\delta \in \mathbb{R}$, is not defined. This curve is defined via the geometric concept known as a retraction \cite{boumal2023,absil2008}. In short, we define the retraction on $\text{EM}(c)^{\frac{1}{2}}$, $R_{\sqrt{p_\theta}} X$ as a curve $\gamma(t)$ that satisfies $\gamma(0)=\sqrt{p_\theta}$ and $\dot{\gamma}(0)=X$ for any $X \in T_{\sqrt{p_\theta}} \text{EM}(c)^{\frac{1}{2}}$ (see the Appendix for more details). Once this is defined, we proceed with the description of the dissimilarity criterion that will be employed to identify the optimal square-root density $\sqrt{p_\theta} \in \text{EM}(c)^{\frac{1}{2}}$ that best approximates the square root of the posterior density $\sqrt{q}$. To achieve this objective, we opt for the $\alpha$-R\'enyi divergence, expressed as $D_\alpha(p \parallel q) = \frac{1}{\alpha-1} \log \E_p \left[ \left(\frac{p}{q}\right)^{\alpha-1} \right]$. When $\alpha = 1$, this divergence is conventionally redefined as the Kullback--Leibler (KL) divergence; i.e., $D_1(p \parallel q) \coloneqq D_{KL}(p \parallel q)$ \cite{vanerven2014}. Notably, the divergence $D_\alpha(p \parallel q)$ exhibits symmetry only when $\alpha = \frac{1}{2}$. Indeed, the case of $\alpha=\frac{1}{2}$ is rather special since $D_{\frac{1}{2}}(q \parallel p)$ is directly related to the Hellinger distance $H(p,q) = \sqrt{1-\exp(-\frac{1}{2}D_{\frac{1}{2}}(q \parallel p))}$ \cite{vanerven2014}. Due to this relation, the fact that $H(p,q) = \frac{1}{\sqrt{2}}\norm{\sqrt{p}-\sqrt{q}}_2$, and that the projection of the square-root density via Fokker--Planck dynamics \eqref{eq:projected_Fokker_Planck} is obtained via a projection on $L^2(\mathcal{X})$ (Lemma 2.1 of \cite{brigo1999}), we will specifically choose $D_{\frac{1}{2}}(q \parallel p_\theta)$ as the loss function to be optimized in the numerical examples in Section \ref{sec:Numerical_Examples}. Specifically, for $\alpha=\frac{1}{2}$, with $q$ given by \eqref{eq:posterior}, the explicit form of $D_{\frac{1}{2}}(q \parallel p_\theta)$ reads
  \begin{gather}
    \begin{aligned}
      D_{\frac{1}{2}}(q \parallel p_\theta) =& -2 \log(\E_{\theta_k^-}\left[\exp(\frac{1}{2} c^\top(\theta_k^- - \theta) - \ell(\cdot,y_k))  \right]) \\
      &+ \left( \psi(\theta_k^-) + Z(\theta_k^-,y_k) - \psi(\theta) \right). \label{eq:D_1_2_explicit}
    \end{aligned}  
  \end{gather}
In what follows, we derive the Riemannian gradient of $D_\alpha(p \parallel q)$, describe the Riemannian gradient descent method for optimization, and verify that the optimization procedure leads to a unique global minimum for any $\alpha \in (0,1)$.

Let the exponential density $p_{\theta_k^-}$ be the predictive density and $q$ be the posterior density as defined in \eqref{eq:posterior}. In what follows, we assume the support of $\exp(-\ell(\cdot,y_k))$ is $\mathcal{D}$. Given a predictive parameter vector $\theta_k^-$, we cast the Bayesian update problem as an optimization problem where we aim to find $\theta \in \Theta$ such that $D_\alpha(p_\theta \parallel q)$ or $D_\alpha(q \parallel p_\theta)$ is minimized. For a regular $\alpha \in (0,1)$, define
\begin{equation}
  A_\alpha(\theta) \coloneqq \E_q \left[\left(\frac{q}{p_\theta}\right)^{\alpha-1}\right] = \E_\theta \left[\left(\frac{q}{p_\theta}\right)^\alpha\right]. \label{eq:A_alpha}
\end{equation}
The following proposition provides the Riemannian gradient of $D_\alpha(q \parallel p_\theta)$ in local coordinates. For brevity, let $\eta(\theta) = \E_\theta[c]$.
\begin{prop}\label{prp:grad_omega_alpha}
  Let $\omega_\alpha(\sqrt{p_\theta}) \coloneqq D_\alpha(q \parallel p_\theta)$, where $\alpha \in (0,1)$. The Riemannian gradient of $\omega_\alpha$ (denoted by $ \mathrm{grad}\omega_\alpha$) can be written in local coordinates as 
\begin{equation}
  \mathrm{grad}\omega_\alpha(\sqrt{p_\theta}) = 4 \sum_{j=1}^m g^{ij}(\theta)[ \eta_j(\theta) - \eta_{\alpha,j}(\theta) ] \partial_i,  \label{eq:grad_omega_alpha}
\end{equation}
where,
\begin{equation}
  \eta_{\alpha,i}(\theta)  \coloneqq  \frac{1}{A_\alpha(\theta)} \E_\theta \left[\left(\frac{q}{p_\theta}\right)^\alpha c_i\right]. \label{eq:c_hat_alpha}
\end{equation}
\end{prop}
\begin{proof}
  By the definition of $\text{grad}\omega_\alpha$,
  \begin{align}
    \inner{\text{grad}\omega_\alpha , \partial_j}_{\sqrt{p_\theta}} =& \frac{1}{\alpha-1} \frac{1}{A_\alpha(\theta)} \int_\mathcal{D} \frac{\partial}{\partial \theta_j} \left(\frac{q^{\alpha-1}}{p_\theta^{\alpha-1}}\right) q \, dx \nonumber\\
    =& -\left[ \frac{1}{A_\alpha(\theta)} \E_\theta \left[\left(\frac{q}{p_\theta}\right)^\alpha c_j\right] - \eta_j(\theta)\right] \nonumber\\
    =& \eta_j(\theta)  - \eta_{\alpha,j}(\theta).
  \end{align}
 
  Therefore, since $\inner{\partial_i, \partial_j} = \frac{1}{4} g_{ij}(\theta)$, we can write the gradient in local coordinates as $\text{grad}\omega_{\alpha}(\sqrt{p_\theta}) = \sum_{i=1}^m w^i \partial_i$, where the component $w^i$ is as follows:
    \begin{align*}
      \inner{\text{grad}\omega_\alpha, \partial_i}_{\sqrt{p_\theta}} =& \sum_{j=1}^m  \inner{w^j \partial_j, \partial_i}_{\sqrt{p_\theta}} = \frac{1}{4} \sum_{j=1}^m w^j g_{ij}(\theta).
    \end{align*}
    Hence, we obtain \eqref{eq:grad_omega_alpha}.
  \end{proof}
  If we opt to optimize the opposite R\'enyi divergence $D_\alpha(p_\theta \parallel q)$, the Riemannian gradient given by \eqref{eq:grad_omega_alpha} can also be used to obtain the Riemannian gradient of $D_\alpha(p_\theta \parallel q)$, since for $\alpha \in (0,1)$, $D_\alpha(p_\theta \parallel q) = \frac{\alpha}{1-\alpha} D_{1-\alpha}(q \parallel p_\theta)$ (Proposition 2 \cite{vanerven2014}). 
  
  We now proceed with the convergence analysis of the local optima of $D_{\alpha}(q \parallel p_\theta)$. For our subsequent analysis, define the set of parameters with vanishing gradient as follows:
  \begin{equation}
      \Theta_{\alpha,\ast} \coloneqq \left\{ \theta \in \Theta : \eta(\theta) = \eta_\alpha(\theta) \right\}. \label{eq:VanishingGradientSet}
  \end{equation}
  By \eqref{eq:grad_omega_alpha} and the minimality of $\text{EM}(c)$, $\theta$ belongs to $\Theta_{\alpha,\ast}$ if and only if $\text{grad}\omega_\alpha(\sqrt{p}_{\theta})=0$.
  
  In the case where the likelihood function $\exp(-\ell(\cdot,y_k))$ is conjugate to the prior \footnote{As in standard statistical definitions, by saying that the likelihood function is conjugate to a prior (in this case the exponential family $\text{EM}(c)$) or that a prior is conjugate to a likelihood function, we mean that the multiplication of any $p_\theta \in \text{EM}(c)$ and the likelihood function $\exp(-\ell(\cdot,y_k))$ is also an element of $\text{EM}(c)$ (after normalization); see, e.g., \cite[Chapter 3]{raiffa2000}.}, then there exists $\theta \in \Theta$ such that $p_\theta = q$ and hence automatically $\eta(\theta) = \eta_\alpha(\theta)$, which leads to $\text{grad}\omega_\alpha = 0$. The following proposition gives the precise statement; see also Section 6.2 of \cite{brigo1999}. 
  \begin{prop}\label{prp:conjugate_condition}
  Let $\theta_0\in \Theta$ and $q = p_{\theta_0}\exp(-\ell(\cdot,y_k) - Z(\theta_0))$, where the support of $\exp(-\ell(\cdot,y_k))$ is also $\mathcal{D}$. For any $\alpha \in (0,1)$, if there exists $\theta_\ell \in \mathbb{R}^m$ where the negative log-likelihood can be written as $-\ell(\cdot,y_k) = c^\top \theta_\ell$, and $Z(\theta_0)<\infty$, then $\theta_\ast = \theta_0 - \theta_\ell \in \Theta_{\alpha,\ast}$ and $D_\alpha(q \parallel p_{\theta_\ast})=0$. 
  \end{prop}
  \begin{proof}
    Observe that in the case where $-\ell(\cdot,y_k) = c^\top \theta_\ell$
    \begin{align*}
      Z(\theta_0) =& \log(\int_{\mathcal{D}} \exp(-c^\top \theta_\ell)\exp(c^\top \theta_0 - \psi(\theta_0)) \, dx)\\
      =& \psi(\theta_0-\theta_\ell) - \psi(\theta_0) < \infty.
    \end{align*}
    Since $\theta_0 \in \Theta$ and $\psi(\theta_0-\theta_\ell) < \infty$, $\theta_0-\theta_\ell = \theta_\ast \in \Theta$. Let 
    \begin{subequations}
      \begin{align}
        \tilde{q}_{\alpha,\theta}\coloneqq& q^\alpha \exp(\tilde{c}^\top\theta - \rho(\theta)), \label{eq:tilde_q}\\
        \rho(\theta) \coloneqq& \log (\int_{\mathcal{D}} \exp(\tilde{c}^\top \theta) q^\alpha \, dx ), \label{eq:rho}
      \end{align}
    \end{subequations}
    where $\tilde{c} \coloneq (1-\alpha)c$. It is straightforward to see that $p_{\theta_\ast} = q = \tilde{q}_{\alpha,\theta_\ast}$, and hence
    \begin{equation}
      \eta(\theta_\ast)-\eta_\alpha(\theta_\ast) = \int_\mathcal{D} c  \left(p_{\theta_\ast} - \tilde{q}_{\alpha,\theta_\ast} \right)\,dx = 0.
    \end{equation}
    Therefore, $\theta_\ast \in \Theta_{\alpha,\ast}$ and $D_\alpha(q \parallel p_{\theta_\ast})=0$. 
  \end{proof}

  The condition for the non-conjugate case is more complex, and the optimal vector parameter $\theta_\ast$ will likely need to be calculated through approximation. However, when $\alpha = 1$, $\eta_\alpha(\theta) = \E_q[c]$; i.e., it does not depend on $\theta$. In this case, if $\Theta_{\alpha,\ast}$ is non-empty, then it contains only one element due to the diffeomorphism between $\theta$ and $\eta(\theta)$; see Theorem 2.2.3 \cite{kass1997}. When $\alpha \neq 1$, the uniqueness does not hold unless additional conditions on the negative log-likelihood are imposed. We will come back to the uniqueness issue shortly. In what follows, we will show that regardless of the uniqueness of the parameter vectors with vanishing gradient, if the density parameters $\theta$ are updated via Riemannian gradient flow \eqref{eq:Riemannian_Gradient_Flow_Constant_Rate}, then $\theta$ approaches $\Theta_{\alpha,\ast}$ asymptotically. 
  \begin{prop}
    Suppose for $\theta \in \Theta$, $p_\theta$ and $q$ have the same support $\mathcal{D}$, where $p_\theta \in \text{EM}(c)$, a minimal exponential family. Let the dynamics of the parameter vector be given by 
    \begin{equation}
        \dv{\theta(t)}{t} = -4 \delta\,g(\theta(t))^{-1} [\eta(\theta(t)) - \eta_{\alpha}(\theta(t)) ], \quad \delta >0.
    \label{eq:Riemannian_Gradient_Flow_Constant_Rate}
    \end{equation}
    For any $\alpha \in [\tfrac{1}{2},1]$, the set $\Theta_{\alpha,\ast}$ defined by \eqref{eq:VanishingGradientSet} is globally asymptotically stable under \eqref{eq:Riemannian_Gradient_Flow_Constant_Rate}.
\end{prop}
\begin{proof}
  Let us define $\tilde{\omega}_\alpha(\theta) := \omega_\alpha(\sqrt{p_\theta}) = D_\alpha(q\parallel p_\theta)$ for any $\theta \in \Theta$, where $\Theta$ is defined in Eq. (3). Since $D_\alpha(q\parallel p) \geq 0$ (Theorem 8 \cite{vanerven2014}), we can choose $\tilde{\omega}_\alpha(\theta)$ as a Lyapunov function candidate. Consider a path $\gamma(t) := \sqrt{p_\theta(t)} \in \text{EM}(c)^{\frac{1}{2}}$ where $\theta(0) \in \Theta$ and the time derivative $\dot{\theta}(t)$ is given by \eqref{eq:Riemannian_Gradient_Flow_Constant_Rate}. Let the level set $\Theta_K := \{\theta \in \Theta : \tilde{\omega}_\alpha(\theta) \leq K \}$. In particular, consider the case with $K = \tilde{\omega}_\alpha(\theta(0))$. 
  First, we claim that $\Theta_K$ is a compact set. To see this, first, the level set $\Theta_K \subset \Theta$ is $m$-dimensional. Hence, to be a compact set, we need to show that it is closed and bounded. The level set $\Theta_K$ is closed since $\tilde{\omega}_\alpha(\theta)$ is continuous on $\theta$ and the preimage of a closed interval $[0, K]$ under $\tilde{\omega}_\alpha$ is also closed. For the boundedness, suppose $\theta_1, \theta_2 \in \Theta_K$. According to the inequality (see eq. 7, Theorem 3 of \cite{vanerven2014}),
  \begin{equation*}
    H(p_{\theta_i}, q)^2 \leq D_\alpha (q \parallel p_\theta), ~\alpha \in [\frac{1}{2},1],
  \end{equation*}
  where $H(p,q)$ is the Hellinger distance between $p$ and $q$. Hence,
  \begin{align*}
    \norm{\sqrt{p_{\theta_1}}-\sqrt{p_{\theta_2}}}_{L^2} \leq& \norm{\sqrt{p_{\theta_1}}-\sqrt{q}}_{L^2} + \norm{\sqrt{p_{\theta_2}}-\sqrt{q}}_{L^2} \\
    =& \sqrt{2} \left[ H(q, p_{\theta_1}) + H(q, p_{\theta_2}) \right] \\
    \leq& \sqrt{2} \left[ \sqrt{D_\alpha (q \parallel p_{\theta_1})} + \sqrt{D_\alpha (q \parallel p_{\theta_2})} \right] \\
    =& 2 \sqrt{2} \sqrt{K}.
  \end{align*}
  Moreover, by the mean-value theorem, there exists $\theta_c$ in the interior of the line connecting $\theta_1$ and $\theta_2$ such that
  \begin{align*}
    &\norm{\sqrt{p_{\theta_1}}-\sqrt{p_{\theta_2}}}_{L^2}^2 =  \int (\sqrt{p_{\theta_1}}-\sqrt{p_{\theta_2}})^2 \, dx  \\
    &= \int (\sqrt{p_{\theta_1}}-( \sqrt{p_{\theta_1}} + \frac{1}{2} \sqrt{p_{\theta_c}} \left( c - \eta(\theta_c) \right)^\top (\theta_1 - \theta_2) ))^2 \, dx  \\
    &= \frac{1}{4} \int p_{\theta_c} (\theta_1 - \theta_2)^\top\left( c - \eta(\theta_c) \right)\left( c - \eta(\theta_c) \right)^\top (\theta_1 - \theta_2) \, dx  \\
    &= \frac{1}{4} (\theta_1-\theta_2)^\top g(\theta_c)(\theta_1-\theta_2) \geq \epsilon \norm{\theta_1 - \theta_2}_2^2
  \end{align*}
  for some $\epsilon > 0$ since $\text{EM}(c)$ is regular. The vector $\theta_c$ belongs to $\Theta$ since $\Theta$ is a convex set. Hence, we have for any $\theta_1, \theta_2 \in \Theta_K$,
  $\norm{\theta_1 - \theta_2}_2 \leq \frac{2\sqrt{2K}}{\sqrt{\epsilon}} < \infty$.
  Next, there exists $V \in T_{\sqrt{p_{\theta(0)}}} \text{EM}(c)^{\frac{1}{2}}$, $V = \dot{\gamma}(0)$ where in the local coordinate $V = \sum w^i \partial_i$ with $w^i = -\delta \sum_{j=1}^m g^{ij}(\theta(0))\left[ \E_{\theta(0)} [c_j] - \eta_{\alpha,j}(\theta(0)) \right]$. In other words, $V = -\delta \left.\text{grad} \omega_\alpha \right|_{\sqrt{p_{\theta(0)}}}$. Therefore, we can write
  \begin{align*}
  \left. \frac{d}{dt} \tilde{\omega}_\alpha \right|_{t=0} &= \left. \frac{d}{dt} \left[ \omega_\alpha (\sqrt{p_{\theta(t)}}) \right]\right|_{t=0}
  = \inner{ \text{grad}\omega_\alpha, V }_{\sqrt{p_{\theta(0)}}} \\
  &= -\delta \norm{\text{grad}\omega_\alpha}^2_{\sqrt{p_{\theta(0)}}}
  \end{align*}
  In this equation, the inner product $\inner{ \text{grad}\omega_\alpha, V }_{\sqrt{p_{\theta(0)}}}$ is the Riemannian inner product with respect to its metric.

  The last exposition makes the set $\Theta_K$ positively invariant with respect to the dynamics of $\theta(t)$. Since the Fisher metric $g(\theta)$ is positive definite for any $\theta \in \Theta$, $\norm{\text{grad}\omega_\alpha}^2_{\sqrt{p_{\theta(0)}}}$ will only be zero at the set of vanishing gradient $\Theta_{\alpha,\ast}$. By the Barbashin-Krasovskii-LaSalle Theorem (Theorem 3.3 \cite{haddad2008}), starting at $\theta(0)\in \Theta_K$, $\theta(t)$ asymptotically approaches $\Theta_{\alpha,\ast}$. Since the negative log-likelihood and the parametric density $p_\theta$ share the same support $\mathcal{D}$, then $D_\alpha(q\parallel p_\theta)<\infty$, and so is $\tilde{\omega}_\alpha$ by definition. Thereby, since $K$ can be chosen arbitrarily large and $\text{EM}^{\frac{1}{2}}(c)$ is open, then the convergence also holds for any $\theta \in \Theta$.
\end{proof}
Next, we present a sufficient condition that will ensure that $\Theta_{\alpha,\ast}$ is a singleton in the case $\alpha \neq 1$. This guarantees that there is a unique global minimum for the optimization of $D_\alpha(q\parallel p_\theta)$. We begin with the following lemma:

\begin{lem}\label{lem:StrictlyConvex}
    Let $\tilde{q}_{\alpha,\theta}$ and $\rho(\theta)$ be defined by \eqref{eq:tilde_q} and \eqref{eq:rho}, respectively. If for all $\theta \in \Theta$,
    \begin{equation}
      \pdv[2]{\psi (\theta)}{\theta} \succ (1-\alpha)\int \left( c - \eta_\alpha(\theta) \right) \left( c - \eta_\alpha(\theta) \right)^\top \tilde{q}_{\alpha,\theta} \, dx,
      \label{eq:positive_definite_condition}
    \end{equation}
    then the $\alpha$-R\'{e}nyi divergence $D_\alpha(q \parallel p_\theta)$ has a positive definite Hessian on $\Theta$.
\end{lem}
\begin{proof}
    We can write 
    \begin{equation*}
      \begin{split}
        D_\alpha(q \parallel p_\theta) &= \frac{1}{\alpha-1} \log \int_{\mathcal{D}} q^\alpha p_\theta^{1-\alpha} \, dx\\
        &= \dfrac{1}{\alpha-1} \log  [\exp((\alpha-1)\psi(\theta)) \int_{\mathcal{D}} q^\alpha \exp((1-\alpha)c^\top \theta) \, dx ]\\
        &= \psi(\theta) - \frac{1}{1-\alpha} \rho(\theta).
      \end{split}
    \end{equation*}
    A straightforward calculation shows that $\frac{1}{1-\alpha}\pdv{\rho(\theta)}{\theta} = \eta_\alpha(\theta)$ and
    \begin{equation*}
      \begin{split}
        \frac{1}{1-\alpha}\pdv[2]{\rho (\theta)}{\theta} 
        &= (1-\alpha)\int \left( c - \eta_\alpha(\theta) \right)\left( c - \eta_\alpha(\theta) \right)^\top \tilde{q}_{\alpha,\theta} \, dx.
      \end{split}
    \end{equation*}
    Hence, if \eqref{eq:positive_definite_condition} holds, the Hessian of $D_\alpha(q \parallel p_\theta)$ is positive definite.
\end{proof}

The following lemma shows the diffeomorphism between $\Theta$ and the set of all possible $\tilde{\eta}_\alpha(\theta)$, which is a subset of $\mathbb{R}^m$, when \eqref{eq:positive_definite_condition} is satisfied. This lemma is essential to prove the uniqueness of the minimizer of $D_\alpha(q \parallel p_\theta)$.

\begin{lem}\label{lem:Diffeomorphism_of_eta_tilde}
    Let $\tilde{\eta}_\alpha \coloneqq \eta(\theta) - \eta_\alpha(\theta)$. When \eqref{eq:positive_definite_condition} is satisfied, the mapping $\theta \;\mapsto\; \tilde{\eta}_\alpha(\theta)$ is a diffeomorphism from $\Theta$ to $\tilde{N}_\alpha \coloneqq \{\tilde{\eta}_\alpha(\theta) : \theta \in \Theta\}$.
\end{lem}
\begin{proof}
    We will use a technique similar to that used in the proof of Theorem 2.2.3 \cite{kass1997}. First, we claim that for any $\theta_1, \theta_2 \in \Theta$,
    \begin{equation}
      (\theta_1-\theta_2)^\top  (\tilde{\eta}_\alpha(\theta_1) - \tilde{\eta}_\alpha(\theta_2) ) \ge 0,
      \label{eq:inequality_in_eta_tilde}
    \end{equation}
    with equality only when $\theta_1 = \theta_2$. By \eqref{eq:positive_definite_condition}, $\pdv[2]{D_\alpha(q \parallel p_\theta)}{\theta}$ is positive definite. Define $K(\lambda) = D_\alpha(q \parallel p_{\lambda \theta_1 + (1-\lambda)\theta_2} )$. Then
    \begin{align*}
      \frac{dK}{d\lambda}
      =& (\theta_1-\theta_2)^\top \tilde{\eta}_\alpha(\lambda \theta_1 + (1-\lambda)\theta_2),
      \\
      \frac{d^2K}{d\lambda^2}
      =& (\theta_1-\theta_2)^\top \dfrac{\partial^2 D_\alpha(q \parallel p_\theta)}{\partial \theta^2} (\theta_1-\theta_2).
    \end{align*}
    Hence $\tfrac{dK}{d\lambda}$ is monotonically increasing in $\lambda$ unless $\theta_1=\theta_2$. Therefore,
    \begin{equation}
      (\theta_1-\theta_2)^\top \tilde{\eta}_\alpha(\theta_2)
      = \left.\frac{dK}{d\lambda}\right|_{\lambda=0}
      < \left.\frac{dK}{d\lambda}\right|_{\lambda=1}
      = (\theta_1-\theta_2)^\top \tilde{\eta}_\alpha(\theta_1),
    \end{equation}
    which implies \eqref{eq:inequality_in_eta_tilde}. If $\tilde{\eta}_\alpha(\theta_1)=\tilde{\eta}_\alpha(\theta_2)$, then $\theta_1=\theta_2$, so the mapping is one-to-one. Smoothness follows from Theorem 2.2.1 of \cite{kass1997}, since $\rho(\theta)$ is strictly convex with derivatives of all orders. Moreover, by the inverse function theorem, the inverse mapping is also smooth.
\end{proof}

  Using Lemma \ref{lem:StrictlyConvex} and Lemma \ref{lem:Diffeomorphism_of_eta_tilde}, we state the following final result regarding the uniqueness of the minimizer of $D_\alpha(q \parallel p_\theta)$.
\begin{prop}\label{prp:minimizer}
  Suppose for $\theta \in \Theta$, $p_\theta$ and $q$ have the same support $\mathcal{D}$, and $p_\theta \in \mathrm{EM}(c)$ is a minimal exponential family. Let \eqref{eq:positive_definite_condition} be satisfied. Then, if the vanishing gradient set $\Theta_{\alpha,\ast}$ is non-empty, it contains only one element, $\theta_\ast$. The corresponding square-root density $\sqrt{p_{\theta_\ast}}$ is a unique global minimizer of $\omega_\alpha(\sqrt{p_\theta})$ for $\alpha \in (0,1)$, defined in Proposition \ref{prp:grad_omega_alpha}.
\end{prop}

\begin{proof}
First, Lemma \ref{lem:StrictlyConvex} shows that $D_\alpha(q \parallel p_\theta)$ is convex. Next, we show that $\omega_\alpha$ is geodesically convex by showing $\mathrm{Hess}\,\omega_\alpha(\sqrt{p_{\theta_0}})\succ 0$ for any $\theta_0 \in \Theta$; see Theorem 11.23 \cite{boumal2023}. Let $R_{\sqrt{p_{\theta_0}}}$ be a retraction on $\mathrm{EM}(c)^{\frac12}$, and let $V = \sum_{i=1}^m v^i \partial_i$. For $t\in\mathbb{R}$, define $\hat{\omega}_\alpha$ on $T_{\sqrt{p_{\theta_0}}}\mathrm{EM}(c)$ by $\hat{\omega}_\alpha(tv)=\omega_\alpha(R_{\sqrt{p_{\theta_0}}}(tV))$. Let $\theta(t)\coloneqq\theta_0+tv$. Then (see Proof of Proposition 6.3 of \cite{boumal2023}),
\begin{align*}
  &\langle \mathrm{Hess}\,\omega_\alpha(\sqrt{p_{\theta_0}})V,V\rangle
   = \langle \mathrm{Hess}\,\hat{\omega}_\alpha(0)v,v\rangle
   = \left.\frac{d^2}{dt^2}\hat{\omega}_\alpha(tv)\right|_{t=0} \\
  &= \frac{d^2}{dt^2}\left. D_\alpha(q\parallel p_{\theta(t)})\right|_{t=0}
   = v^\top\pdv[2]{D_\alpha(q\parallel p_\theta)}{\theta}v > 0.
\end{align*}

Since $\omega_\alpha$ is geodesically convex, any local minimizer of $\omega_\alpha$ will be the global minimizer; see Theorem 11.6 \cite{boumal2023}. By Propositions 6.3 and 6.5 of \cite{boumal2023}, $\sqrt{p_{\theta_\ast}}$ is a local minimizer of $\omega_\alpha(\sqrt{p_{\theta_\ast}})$ if it is a second-order critical point for $\omega_\alpha$. This necessitates that $\mathrm{grad}\,\omega_\alpha=0$ and the Riemannian Hessian of $\omega_\alpha$ is semi-positive definite at $\sqrt{p_{\theta_\ast}}$, i.e., $\mathrm{Hess}\,\omega_\alpha(\sqrt{p_{\theta_\ast}})\succcurlyeq0$. Since $\omega_\alpha$ is geodesically convex, $\mathrm{Hess}\,\omega_\alpha(\sqrt{p_{\theta_\ast}})\succcurlyeq0$. Given that $\mathrm{EM}(c)$ is minimal and, hence, by Theorem \ref{thm:kass_exponential_family_expectation}, $g(\theta)\succ0$, it follows from Proposition \ref{prp:grad_omega_alpha} that the condition $\tilde{\eta}(\theta)=0$ guarantees $\mathrm{grad}\,\omega_\alpha=0$. The uniqueness of $\sqrt{p}_{\theta_\ast}$ follows from Lemma \ref{lem:Diffeomorphism_of_eta_tilde}, since if $0\in\tilde{N}_\alpha$, then $\theta_\ast$ is unique, as is its mapping $\sqrt{p_{\theta_\ast}}$. This completes the proof.
\end{proof}

An immediate observation from Proposition \ref{prp:minimizer} is as follows. When $\alpha\to1$, the inequality \eqref{eq:positive_definite_condition} holds automatically, regardless of $\ell(\cdot,y_k)$. Therefore, the uniqueness of the minimum of $D_{KL}(q\parallel p_\theta)$ holds irrespective of $\ell(\cdot,y_k)$. From this angle, not only does Proposition \ref{prp:minimizer} generalize a previous result on the solution of minimization of KL divergence $D_{KL}(q\parallel p_\theta)$ via moment-matching, but it also asserts the uniqueness of the solution to $\eta(\theta)=\eta_\alpha(\theta)$ and that it is a global minimum; see, for example, \cite[Theorem 1]{herbrich2005}.

Lastly, as we mentioned earlier, in the numerical experiments below, we will opt to minimize $D_{\frac12}(q\parallel p_\theta)$. In Section \ref{sec:Numerical_Examples}, we demonstrate that choosing this divergence rather than $D_{KL}(q\parallel p_\theta)$ turns out to be beneficial, as it leads to an approximated posterior that has a smaller Hellinger distance to the posterior density.

  \section{Numerical Implementation}\label{sec:numerical_implementation}

To minimize the R\'enyi divergence, we use the Riemannian gradient descent algorithm. Although more complex optimization algorithms tailored for Riemannian geometry exist, we opt for Riemannian gradient descent in this context because it avoids the need for affine connections or geodesics from the Riemannian manifold; see \cite{boumal2023}. During the implementation, all expectations and the cumulant-generating function evaluations in \eqref{eq:Riemannian_Gradient_Flow_Constant_Rate} will be approximated. Specifically, to calculate the cumulant-generating function $\psi(\theta)$ and expectations with respect to $p_\theta$, we use the adaptive Gaussian-based bijection from \cite{emzir2023a} that maps quadrature nodes in the domain $\mathcal{D}_c \coloneqq (-1,1)^d$ to adaptively cover the high-density region of $p_\theta$. 

A brief description of the sparse-quadrature scheme used in this work is as follows. Let $d$-dimensional numerical quadrature with $N$ quadrature nodes be defined as follows:
\begin{equation}
  Q^{d}_N  [ \varphi  ] \coloneqq \int_{\mathcal{D}_c} \varphi(\tilde{x}) \, d\tilde{x} \approx \sum_{i=1}^N w_i \, \varphi(\tilde{x}_i),
  \label{eq:Numerical_integration}
\end{equation}
for a test function $\varphi: \mathcal{D}_c \to \mathbb{R}$, $\{\tilde{x}_i\}_{i=1}^N \in \mathcal{D}_c$ are the nodes, and $\{w_i\}_{i=1}^N$ are the weights of the quadrature rule. Since the approximated filtering density $p_{\theta_t}$ moves with time, the quadrature nodes $\{\tilde{x}_i\}_{i=1}^N$ need to be adaptively updated. To force the quadrature nodes to move to a region covering the high-density region of $p_{\theta_t}$, the adaptive bijection proposed in \cite{emzir2023a} is formulated as follows:
\begin{align}
  \beta_\xi(\tilde{x}) = \mu + \sqrt{2}\,T^{-1}\Lambda^{1/2}\,\erf^{-1}(\tilde{x}), \quad \tilde{x}\in \mathcal{D}_c, \label{eq:phi_xi_gaussian}
\end{align}
where $\erf^{-1}(\tilde{x}) = [\erf^{-1}(\tilde{x}_1),\ldots,\erf^{-1}(\tilde{x}_d)]^\top$ and $\erf^{-1}$ is the inverse of the error function. In \eqref{eq:phi_xi_gaussian}, the bijection’s parameters are given by $\xi = (\mu, \Sigma)$, where $\mu = \E_\theta[x]$, $\Sigma = \E_\theta[(x-\mu)(x-\mu)^\top]$, and $\Lambda$ and $T$ are obtained from the eigendecomposition of $\Sigma$, i.e., $\Sigma = T^{-1}\Lambda T$. 

It has been shown that for certain applications, the bijection \eqref{eq:phi_xi_gaussian} is more efficient in terms of accuracy per number of quadrature nodes compared to the sparse Gauss--Hermite quadrature (sGHQ) \cite{emzir2023a}. However, during our implementation, we noticed that using the bijection \eqref{eq:phi_xi_gaussian} requires a significantly higher computational cost than the sGHQ. Therefore, we modify the domain of the quadrature nodes to be $\mathbb{R}^d$ by replacing $\tilde{x}_i$ with $\tilde{y}_i := \erf^{-1}(\tilde{x}_i)$. Furthermore, we replace $T^{-1}\Lambda^{1/2}$ with $L = \mathrm{chol}(\Sigma)$ for efficiency, where $\mathrm{chol}(\Sigma)$ is the Cholesky factorization of $\Sigma$. The parameter of the bijection becomes $\xi=(\mu,L)$, and the bijection reads
\begin{align}
  \beta_\xi(\tilde{y}) = \mu + \sqrt{2}\,L\,\tilde{y}, \quad \tilde{y}\in \mathbb{R}^d. \label{eq:phi_xi_gaussian_modified}
\end{align}
The illustration of the bijection \eqref{eq:phi_xi_gaussian_modified} is shown in Figure \ref{fig:explaining_sparse_quadrature}.

\begin{figure}[!h]
  \centering
  \includegraphics[width=\linewidth]{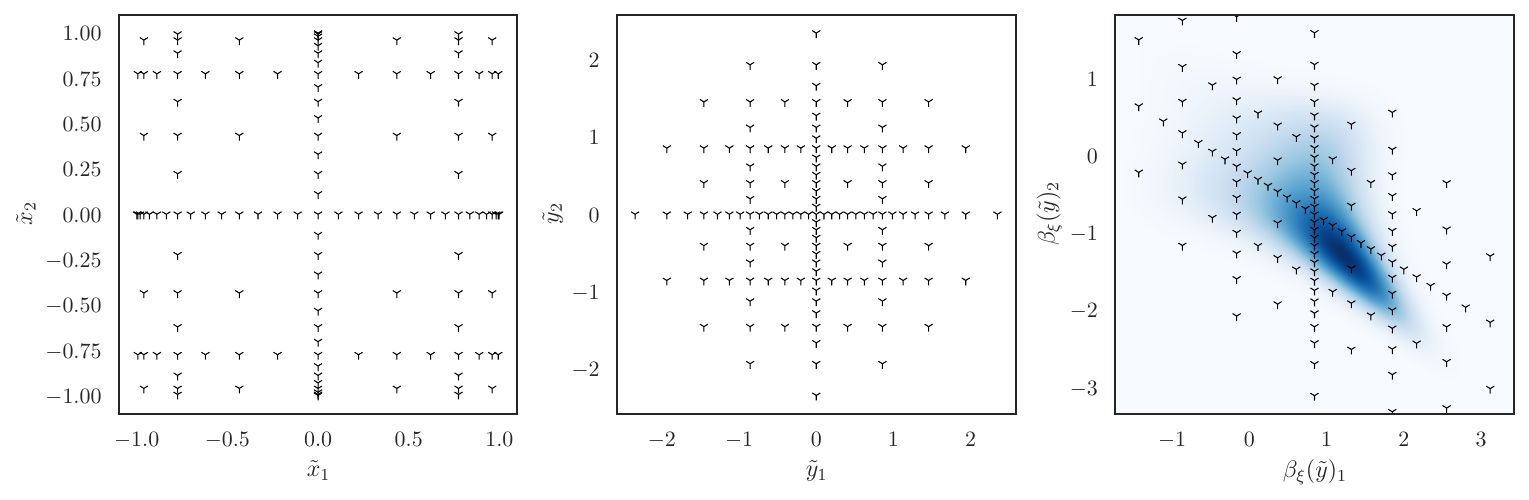}
  \caption{On the left is the scatter plot of the original fourth-level Gauss–Patterson quadrature nodes in $\mathcal{D}_c$ \cite{bungartz2004}. At the center is the scatter plot of $\tilde{y}_i = \erf^{-1}(\tilde{x}_i)$. The right plot shows the transformed quadrature nodes $\beta_\xi(\tilde{y}_i)$, where $\xi = (\mu,L)$ with $\mu = \E_\theta[x]$ and $L = \mathrm{chol}(\Sigma)$, and $\Sigma = \E_\theta[(x-\mu)(x-\mu)^\top]$ for some $p_\theta\in\mathrm{EM}(c)$. Notice how the quadrature nodes cover the high-density region of $p_\theta$.}
  \label{fig:explaining_sparse_quadrature}
\end{figure}

Using bijection \eqref{eq:phi_xi_gaussian_modified}, we define for a function $\varphi:\mathbb{R}^d\to\mathbb{R}$:
\begin{equation}
  Q^d_N  [\varphi ] \coloneqq \int_{\mathbb{R}^d} \varphi(\tilde{y}) \, d\tilde{y}
  \approx \sum_{i=1}^N w_{s,i}\,\varphi(\tilde{y}_i),
  \label{eq:Numerical_integration_modified}
\end{equation}
where $w_{s,i} =  (\tfrac{1}{2\sqrt{\pi}} )^d \exp(\|\tilde{y}_i\|^2)\,w_i$. The cumulant-generating function is approximated by
\begin{align}
  \psi^N(\theta) \coloneqq \log (Q^d_N[\exp(c^\top\beta_\xi(\tilde{y})\,\theta)\,2^{\frac{d}{2}}\det(L)] ).
  \label{eq:psi_N}
\end{align}
Furthermore, we approximate the expectation of any function $\varphi:\mathbb{R}^d\to\mathbb{R}$ with respect to $p_\theta$ by
\begin{equation}
  \E_\theta[\varphi]^{(N)} \coloneqq Q^d_N [\varphi(\beta_\xi(\tilde{y}))\exp(c^\top\beta_\xi(\tilde{y})\,\theta-\psi^N(\theta))\,2^{\frac{d}{2}}\det(L) ].
  \label{eq:Numerical_expectation}
\end{equation}

During the implementation, the nodes $\{ \tilde{y}_i\}_{i=1}^N$ are computed once and saved in memory, and used to calculate $\psi^N(\theta)$ and all approximated expectations via \eqref{eq:psi_N} and \eqref{eq:Numerical_expectation}, respectively. Using this sparse-grid quadrature setup, the single-step implementation of the Riemannian gradient descent is given in Algorithm \ref{alg:riemannian_gradient_descent}, while the overall single-step projection filter procedure appears in Algorithm \ref{alg:projectionfiler}. We select the natural statistics as $c = \{x^{\mathbf{i}}\}$, where $2 \leq |\mathbf{i}| \leq n_o$ and $\mathbf{i} \in \mathbb{N}_0^d$ is a multi-index.  Therefore, natural statistics vector elements are linearly independent and also include $x_i$, $x_i x_j$ for $i,j = 1, \dots, d$. Due to this inclusion, the bijection parameters can be calculated directly from the natural statistics expectations: there exist $T_\mu \in \mathbb{R}^{d \times m}$ and a linear map $\Phi_\Sigma: \mathbb{R}^{m} \to \mathbb{R}^{d \times d}$ such that $\mu = T_\mu \eta(\theta)$ and $\expvalb{\theta}{xx^\top} = \Phi_\Sigma(\eta(\theta))$. Given $\eta(\theta)$ the bijection parameters $\mu$ and $L$ can be computed as follows
\begin{gather}
  \begin{aligned}
  \mu  =& T_\mu \eta, &  \Sigma =& \Phi_\Sigma(\eta) - \mu \mu^\top, &  L =& \mathrm{chol}(\Sigma).
\end{aligned} \label{eq:parameters_update}
\end{gather}  

The Fisher metric inverse $g^{-1}$ in Algorithm \ref{alg:riemannian_gradient_descent} is guaranteed to exist by Theorem \ref{thm:kass_exponential_family_expectation} since the exponential family is minimal. However, the positive definiteness of $g$ might be violated in practice due to its approximation via the sparse-grid quadrature. In this case, a higher sparse-quadrature level or an adaptive Tikhonov regularization can be used to enforce the positive definiteness of $g$.

\begin{algorithm}
  \caption{Single-Step Riemannian Gradient Descent}\label{alg:riemannian_gradient_descent}
  \begin{algorithmic}[1]
  \Procedure{RiemannianGradientDescent}{$\theta,\xi,\theta_k^-,y_k,\delta$}
    \State $\eta_\alpha - \eta \gets \pdv{D_{\frac{1}{2}}(q\parallel p_\theta)}{\theta}  [\theta,\xi,\theta_k^-,y_k]$ \Comment{Automatic Differentiation of \eqref{eq:D_1_2_explicit}}
    \State $g\gets \pdv[2]{\psi^N(\theta)}{\theta}$   \Comment{Automatic Differentiation of \eqref{eq:psi_N}}
    \State $\theta \gets \theta - 4\delta\,g^{-1}(\eta-\eta_\alpha)\,dt$ \Comment{\eqref{eq:Riemannian_Gradient_Flow_Constant_Rate}}
    \State $\eta \gets \pdv{\psi^N(\theta)}{\theta}$   \Comment{Automatic Differentiation of \eqref{eq:psi_N}}
    \State $\xi \gets$ \Call{NewBijectionParams}{$\eta$} \Comment{ Update bijection parameter using \eqref{eq:parameters_update}}
    \State \textbf{return} $\theta,\xi$
  \EndProcedure
  \end{algorithmic}
\end{algorithm}

\begin{algorithm}
  \caption{Single-Step Projection Filter Using Parametric Bijection}\label{alg:projectionfiler}
  \begin{algorithmic}[1]
    \Procedure{ProjectionFilter}{$\theta_{k-1},\xi_{k-1},y_k,N_T,\delta$}
      \State $\theta_k^-,\xi_k^- \gets$ \Call{PredictiveUpdate}{$\theta_{k-1},\xi_{k-1}$} \Comment{Propagating ODE \eqref{eq:projected_Fokker_Planck} from $t = (k-1)\Delta t$ to $t = k \Delta t$}
      \State $\theta_k,\xi_k \gets \theta_k^-,\xi_k^-$
      \For{$j = 1,\dots,N_T$}
        \State $\theta_k,\xi_k \gets$ \Call{RiemannianGradientDescent}{$\theta_k,\xi_k,\theta_k^-,y_k,\delta$}
      \EndFor
      \State \textbf{return} $\theta_k,\xi_k$
    \EndProcedure
  \end{algorithmic}
\end{algorithm}

\section{Numerical Examples}\label{sec:Numerical_Examples}
To demonstrate the effectiveness of the proposed method, we apply this technique to two Bayesian update problems. We utilize the numerical package for the projection filter available from \url{https://github.com/puat133/Correlated_Noise_Projection_Filter}. For the sparse-grid quadrature, we employ the Gauss--Patterson sparse grid \cite{bungartz2004} with level 6, which corresponds to 769 nodes, modifying the quadrature nodes and weights according to \eqref{eq:phi_xi_gaussian_modified} and \eqref{eq:Numerical_integration_modified}, respectively. We also compute the approximated posteriors obtained by minimizing $D_{KL}(q \parallel p_\theta)$. In the following numerical examples, we calculate the posterior density $q$ and expectations with respect to it using the same sparse-grid settings. Specifically, we expand the exponential family to $\text{EM}(\tilde{c})$, adding $c_{m+1} = \ell(\cdot,y_k)$ as the $(m+1)$-th natural statistic, i.e., $\tilde{c} = [c_1, \ldots, c_m, \ell(\cdot,y_k)]^\top$. Consequently, the natural parameters of $q$ are defined as $\tilde{\theta} = [\theta_{k,1}^-, \ldots, \theta_{k,m}^-, -1]^\top$; see Proposition 5.2 of \cite{emzir2023a}. In this section, our focus is solely on the Bayesian update part. Therefore, it is assumed that the state $x_t$ is two-dimensional, and that after the predictive update via the projection of the square root of the Fokker--Planck equation given by \eqref{eq:projected_Fokker_Planck}, the parametric density is given by $p_{\theta_k^-} = p_{\theta_0} = \mathcal{N}(\mu, \Sigma)$ for some $\mu \in \mathbb{R}^2, \Sigma \in \mathbb{R}^{2\times 2}$.

\subsection{A Multimodal Two-Dimensional Case}\label{sec:example_1}
For the first numerical example, we choose $\ell(x_k,y_k) = 0.5 \norm{\frac{\sin(x_k-y_k)}{\sigma_y}}^2$, with $\sigma_y=\frac{1}{2}$, and the predictive density is set to $p_{\theta_k^-} \coloneqq p_{\theta_0} = \mathcal{N}\left([1,1]^\top,I \right)$. The negative log-likelihood $\ell(\cdot,y_k)$ is highly nonlinear, and the corresponding posterior $q$ is multimodal. We choose the maximum order of monomials in the natural statistics $c$ to be four, $n_o=4$, and $y=[0,0]^\top$, and $dt=1.25 \times 10^{-2}$ for $N_t=400$ iterations. Figure \ref{fig:density_comparison_particle_simple_4_EXNO_1} shows that at the end of the simulation time $T=N_t dt$, the approximated posterior $p_{\theta_T}$ obtained by minimizing $D_{\frac{1}{2}}(q \parallel p_\theta)$ resulted in a closer resemblance compared to $p_{\theta_T}^{KL}$, the one obtained by minimizing $D_{KL}(q \parallel p_\theta)$. The $p_{\theta_T}^{KL}$ is wider and has a lower peak compared to $p_{\theta_T}$. Moreover, the two minor modes on the top and the right of the major mode are less separated in $p_{\theta_T}^{KL}$. We can also see from Figure \ref{fig:Hell_KL_4_EXNO_1} that $p_{\theta_T}$ has a substantially lower Hellinger distance compared to $p_{\theta_T}^{KL}$ ($H(q, p_{\theta_T})=1.066 \times 10^{-1}$ and $H(q, p_{\theta_T}^{KL})=1.296 \times 10^{-1}$). As expected, the Riemannian gradient descent method that minimizes $D_{\frac{1}{2}}(q\parallel p_\theta)$ produces a posterior approximate with $D_{KL}(q\parallel p_{\theta_T})$ ($1.487 \times 10^{-1}$) higher than those of $D_{KL}(q\parallel p_{\theta})$ minimization ($1.125\times 10^{-1}$). Nonetheless, this highlights the benefit of minimizing $D_{\frac{1}{2}}(q \parallel p_\theta)$ rather than $D_{KL}(q \parallel p_\theta)$. Increasing $n_o$ to $6$ for both approximations produces approximated posterior densities where the gap between the two Hellinger distances becomes smaller ($H(q, p_{\theta_T})=7.951 \times 10^{-2}$ and $H(q, p_{\theta_T}^{KL})=8.570 \times 10^{-2}$). We also report that applying the ordinary gradient descent (that is using the Euclidean gradient, rather than the Riemannian gradient) to this example resulted in a numerical failure. This clearly shows the merit of using the Riemannian gradient descent method.

\begin{figure}[h]
  \centering
  \includegraphics[trim={0.8cm 3.0cm 1.0cm 3.0cm},clip,width=\linewidth]{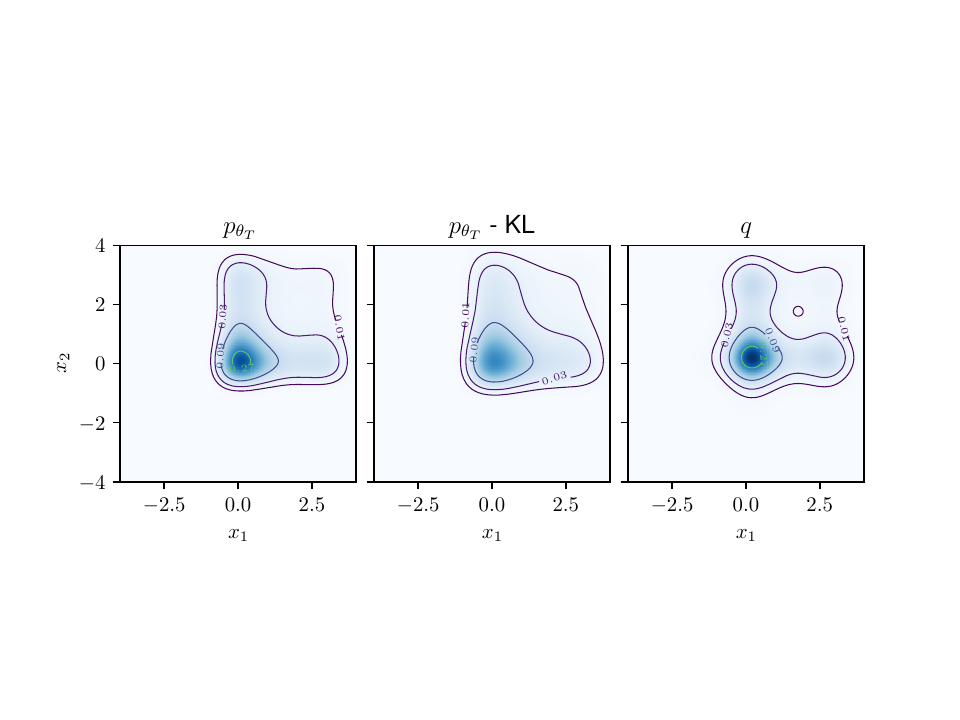}
  \caption{Comparison of the approximated posteriors for the numerical example in Section \ref{sec:example_1}. The approximated posterior on the left is obtained by minimizing $D_{\frac{1}{2}}(q\parallel p_\theta)$, the one in the center by minimizing $D_{KL}(q\parallel p_\theta)$. The actual posterior density $q$ is shown on the right.}
  \label{fig:density_comparison_particle_simple_4_EXNO_1}
\end{figure}

\begin{figure}[!h]
  \centering
  \includegraphics[trim={0.25cm 3.25cm 0.0cm 4.25cm},clip,width=\linewidth]{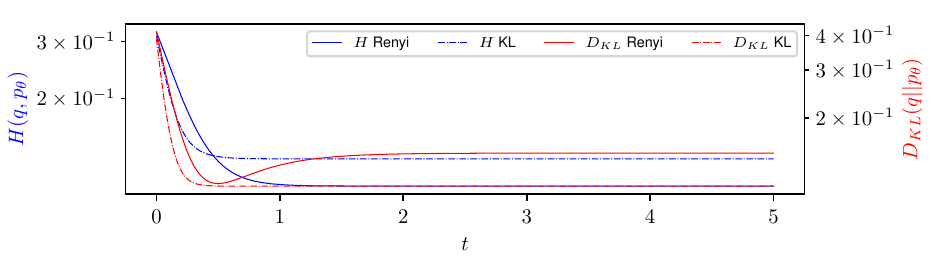}
  \caption{Evolution of Hellinger distances and KL-divergence for the numerical example in Section \ref{sec:example_1}. The straight lines correspond to $H(p,q)$ (blue) and $D_{KL}(q\parallel p_\theta)$ (red), respectively, where the approximated posterior is solved by minimizing $D_{\frac{1}{2}}$, while the dashed lines correspond to those quantities where the approximated posterior is solved by minimizing $D_{KL}$.}
  \label{fig:Hell_KL_4_EXNO_1}
\end{figure}

\subsection{Two-Dimensional Tracking Problem}\label{sec:example_2}
For the second example, we apply this method to a case where $\ell(x_k,y_k) = \frac{1}{2} (y_k-h(x_k))^\top R^{-1}(y_k-h(x_k))$, with:
  \begin{equation}
    h(x) = \mqty[\sqrt{x_1^2+x_2^2+z_0^2},&
                  \tan^{-1}(\frac{x_1}{x_2}),&
                \tan^{-1}(\frac{z_0}{\norm{x}})]^\top.  
  \end{equation}
  The function $h$ is a commonly used measurement function for target tracking problems where the first measurement is the distance, and the last two measurements are the azimuth and elevation angles. We test the case where $z_0=0.2$, $R = \text{diag}([2 \times 10^{-2},4 \times 10^{-1},4 \times 10^{-1}])$. Here, we set $p_{\theta_0}=\mathcal{N}(\mu,\Sigma)$, where $\mu=[\frac{1}{2},-\frac{1}{2}]^\top, \Sigma = 5 \times 10^{-2} I$, and $y=h(\mu)$. For this example, we choose $n_o=2$, and $dt=5 \times 10^{-2}$ for $N_t=100$ iterations. The results of this example can be seen in Figures \ref{fig:density_comparison_particle_simple_2_EXNO_0} and \ref{fig:Hell_KL_2_EXNO_0}. Using similar legends as in Section \ref{sec:example_1}, Figure \ref{fig:density_comparison_particle_simple_2_EXNO_0} shows that the approximated posterior $p_{\theta_T}$ has a closer resemblance compared to $p_{\theta_T}^{KL}$, which tends to be wider and has a lower peak compared to the former. As for the Hellinger distance, we can see from Figure \ref{fig:Hell_KL_2_EXNO_0} that $p_{\theta_T}$ has a substantially lower Hellinger distance compared to $p_{\theta_T}^{KL}$. Again, this emphasizes the benefit of minimizing $D_{\frac{1}{2}}(q \parallel p_\theta)$ rather than $D_{KL}(q \parallel p_\theta)$. Note, however, as the posterior $q$ is only single-mode, when we increase $n_o$ to $4$, both $p_{\theta_T}$ and $p_{\theta_T}^{KL}$ are almost indistinguishable. If we rather use the Euclidean gradient descent to optimize $D_{\frac{1}{2}}$ or KL-divergence, the decreases are given in Figure \ref{fig:Hell_KL_Euclidean_2_EXNO_0}. The declines are significantly smaller compared to the ones in Figure \ref{fig:Hell_KL_2_EXNO_0}.

\begin{figure}[h]
  \centering
  \includegraphics[trim={0.7cm 3.0cm 1.0cm 3.0cm},clip,width=\linewidth]{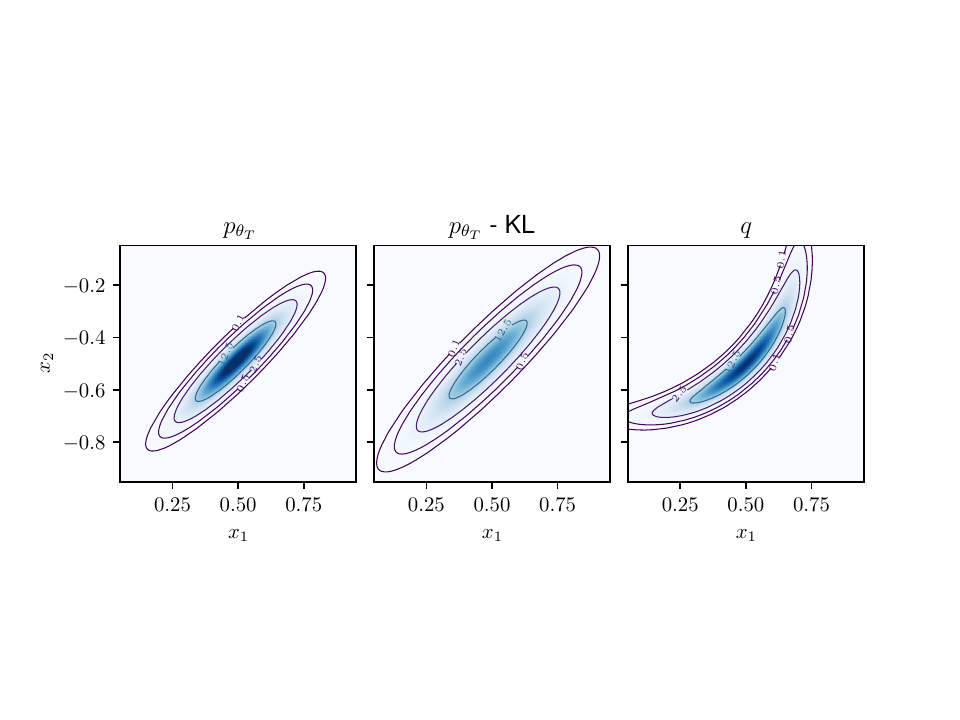}
  \caption{Comparison of the approximated posteriors with $n_o=2$ for the numerical example in Section \ref{sec:example_2}. The legend is similar to Figure \ref{fig:density_comparison_particle_simple_4_EXNO_1}.}
  \label{fig:density_comparison_particle_simple_2_EXNO_0}
\end{figure}

\begin{figure}[!h]
  \centering
  \includegraphics[trim={0.25cm 3.25cm 0.0cm 4.25cm},clip,width=\linewidth]{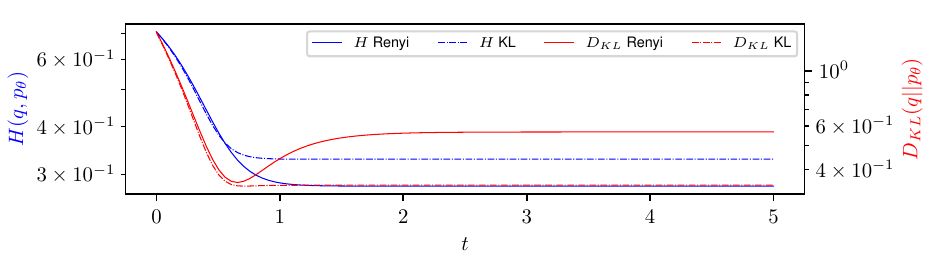}
  \caption{Evolution of Hellinger distances and KL-divergence for the numerical example in Section \ref{sec:example_2}. The legend is similar to Figure \ref{fig:Hell_KL_4_EXNO_1}.}
  \label{fig:Hell_KL_2_EXNO_0}
\end{figure}

\begin{figure}[!h]
  \centering
  \includegraphics[width=\linewidth]{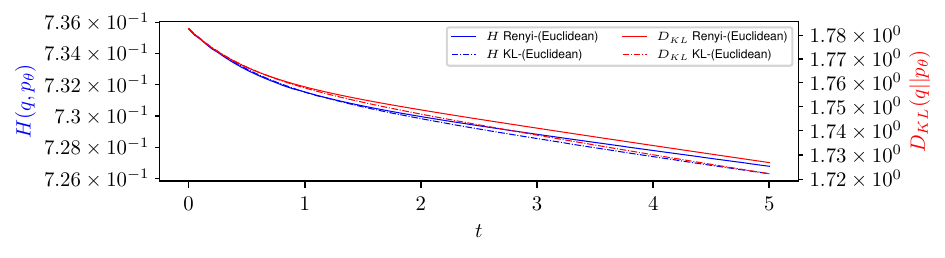}
  \caption{Evolution of Hellinger distances and KL-divergence for the numerical example in Section \ref{sec:example_2}, solved using the Euclidean gradient. The legend is similar to Figure \ref{fig:Hell_KL_4_EXNO_1}.}
  \label{fig:Hell_KL_Euclidean_2_EXNO_0}
\end{figure}

\subsection{Comparison Against Other Bayesian Update Approximation Methods}

  \begin{table}[h]
    \centering
    \scriptsize
    \begin{tabular}{|l|r|r|r|r|}
      \hline
      \textbf{Method} & \textbf{iter.} & \textbf{Hell. Dist.} & \textbf{FLOP} & \textbf{Time(s)} \\
      \hline
      Unscented & 1 & 3.187$\times 10^{-1}$ & 3.510$\times 10^{2}$ & 9.310$\times 10^{-4}$ \\
      GH -order 17 & 1 & 3.207$\times 10^{-1}$ & 1.398$\times 10^{4}$ & 6.599$\times 10^{-4}$ \\
      R-0.5 order 2 - Euler & 50 & 3.083$\times 10^{-1}$ & 5.700$\times 10^{5}$ & 9.012$\times 10^{-3}$ \\
      R-0.5 order 4 - Euler & 100 & 1.080$\times 10^{-1}$ & 1.947$\times 10^{6}$ & 5.953$\times 10^{-2}$ \\
      R-0.5 order 8 - Tsit5 & 50 & 2.491$\times 10^{-2}$ & 3.182$\times 10^{7}$ & 1.720$\times 10^{-1}$  \\
      KL order 2 - Euler & 50 & 3.094$\times 10^{-1}$ & 5.154$\times 10^{5}$ & 5.604$\times 10^{-3}$  \\
      KL order 4 - Euler & 100 & 1.302$\times 10^{-1}$ & 1.835$\times 10^{6}$ & 3.223$\times 10^{-2}$  \\
      KL order 8 - Tsit5 & 50 & 3.173$\times 10^{-2}$ & 3.118$\times 10^{7}$ & 9.785$\times 10^{0}$  \\
      Particle 4.8$\times 10^{4}$ smpls & 1 & 2.414$\times 10^{-1}$ & 2.449$\times 10^{6}$ & 1.266$\times 10^{-2}$  \\
      Particle 4.8$\times 10^{5}$ smpls & 1 & 8.271$\times 10^{-2}$ & 2.448$\times 10^{7}$ & 2.835$\times 10^{-2}$  \\
      Particle 4.8$\times 10^{6}$ smpls & 1 & 2.931$\times 10^{-2}$ & 2.448$\times 10^{8}$ & 1.910$\times 10^{-1}$  \\
      Particle 4.8$\times 10^{7}$ smpls & 1 & 1.250$\times 10^{-2}$ & 2.448$\times 10^{9}$ & 1.589$\times 10^{0}$  \\
      \hline
    \end{tabular}
    \caption{FLOP and execution time comparison for example Section VI.A.}
    \label{tab:example_VI_A}  
  \end{table}

  \begin{table}[h]
  \centering
  \scriptsize
  \begin{tabular}{|l|r|r|r|r|}
    \hline
    \textbf{Method} & \textbf{iter.} & \textbf{Hell. Dist.} & \textbf{FLOP} & \textbf{Time(s)} \\
    \hline
    Unscented & 1 & 4.488$\times 10^{-1}$ & 5.680$\times 10^{2}$ & 1.822$\times 10^{-3}$ \\
    GH -order 17 & 1 & 4.636$\times 10^{-1}$ & 2.385$\times 10^{4}$ & 1.091$\times 10^{-3}$ \\
    R-0.5 order 2 - Euler & 50 & 2.795$\times 10^{-1}$ & 5.860$\times 10^{5}$ & 9.742$\times 10^{-3}$ \\
    R-0.5 order 4 - Euler & 100 & 7.916$\times 10^{-2}$ & 1.963$\times 10^{6}$ & 5.442$\times 10^{-2}$ \\
    R-0.5 order 6 - Tsit5 & 50 & 2.011$\times 10^{-2}$ & 3.475$\times 10^{7}$ & 8.051$\times 10^{-1}$ \\
    KL order 2 - Euler & 50 & 3.162$\times 10^{-1}$ & 5.314$\times 10^{5}$ & 5.161$\times 10^{-3}$ \\
    KL order 4 - Euler & 100 & 2.577$\times 10^{-1}$ & 1.851$\times 10^{6}$ & 3.154$\times 10^{-2}$ \\
    KL order 6 - Tsit5 & 50 & 2.207$\times 10^{-2}$ & 3.367$\times 10^{7}$ & 1.609$\times 10^{0}$ \\
    Particle 4.8$\times 10^{4}$ smpls & 1 & 2.180$\times 10^{-1}$ & 3.457$\times 10^{6}$ & 1.261$\times 10^{-2}$ \\
    Particle 4.8$\times 10^{5}$ smpls & 1 & 6.365$\times 10^{-2}$ & 3.456$\times 10^{7}$ & 3.131$\times 10^{-2}$ \\
    Particle 4.8$\times 10^{6}$ smpls & 1 & 2.153$\times 10^{-2}$ & 3.456$\times 10^{8}$ & 2.525$\times 10^{-1}$ \\
    Particle 4.8$\times 10^{7}$ smpls & 1 & 8.644$\times 10^{-3}$ & 3.456$\times 10^{9}$ & 2.582$\times 10^{0}$\\
    \hline
  \end{tabular}
  \caption{FLOP and execution time comparison for example Section VI.B.}
  \label{tab:example_VI_B}
\end{table}

We compare the performance and computational cost of our method to those of the Bayesian update approximations done via two sigma-point methods, the unscented and Gauss-Hermite sigma points \cite{julier1997,ito2000,sarkka2013}, as well as the systematic resampling method for the particle filter \cite{chopin2020}. The ground truth posterior density $q$ was calculated using a Gauss--Kronrod sparse grid, with the sparse grid level 9. To calculate the Hellinger distance between $q$ and the results of the resampling method, we created a two-dimensional histogram on $500 \times 500$ grid points from the resampled particles. Then the Hellinger distance is calculated via numerical integration. For the sigma point methods, the posterior mean and covariance obtained from the sigma-point update method are used to get the corresponding natural parameters, which are then used to calculate the Hellinger distance.

We also complement our comparison with the numerical result from solving the Riemannian gradient flow using a higher-order solver (Tsitouras' 5/4 method \cite{tsitouras2011}) via the diffrax package \cite{kidger2022}. All numerical implementations are performed using JAX \cite{bradbury2018}. 
The floating point operations' counting was done via the \verb|cost_analysis| method from the \verb|jax.stages.Compiled| class. For these comparisons, we use a Gauss--Kronrod sparse grid for the numerical integration with varying levels of accuracy. The comparison results for examples in Sections VI.A and VI.B are given in Tables \ref{tab:example_VI_A} and \ref{tab:example_VI_B}, respectively.
      
From these tables, we can highlight a few things. Upon selecting the maximum order of polynomial $n_o$ equal to two (Gaussian family case), the approximated posterior obtained via the optimization of $\frac{1}{2}$-R\'enyi divergence outperforms the unscented and Gauss-Hermite posterior density approximations, and does so with a significant margin, as seen in Tables \ref{tab:example_VI_A} and \ref{tab:example_VI_B}. However, this comes with a significantly higher computational cost compared to the unscented and Gauss-Hermite transformations. Increasing $n_o$ to four greatly decreases the Hellinger distances in both examples, making the Hellinger distances comparable to those of the resampling method with $4.8 \times 10^5$ samples. Reaching this level of Hellinger distance is impossible with Gaussian approximation. In this case, the resampling method requires ten times the floating point operations compared to the $\frac{1}{2}$-R\'enyi divergence optimization. For the example in Section VI.A, we were able to produce the result with $n_o=8$, where the Hellinger distance to the posterior is $2.491 \times 10^{-2}$, which is smaller than that of the resampling method with $4.8 \times 10^6$ samples. This time, the resampling method requires about eight times the floating point operations compared to the $\frac{1}{2}$-R\'enyi divergence optimization. In general, our Bayesian update approximation method offers a unique balance between the rough Gaussian-based approximations like the sigma-point based methods, which are computationally cheap, and the particle-based approximations, where the approximation can be made as accurate as possible at a very high computational cost.

\section{Conclusions}\label{sec:Conclusions}
We have formulated the Bayesian update step of exponential family projection filters for continuous-discrete problems with non-conjugate priors via a Riemannian optimization procedure applied to $\frac{1}{2}$-R\'enyi divergence on the $\text{EM}(c)^{\frac{1}{2}}$ manifold. We chose this particular divergence order to ensure compatibility with the projection of the Fokker--Planck equation in the prediction step. We also proved that if a point $p \in \text{EM}(c)^{\frac{1}{2}}$ satisfies a certain moment-matching criterion, then it is the local minimum of $\alpha$-R\'enyi divergence. By implementing an Euler approximation to the Riemannian gradient flow, we show the effectiveness of this method against the standard Riemannian $D_{KL}$ optimization to approximate highly non-Gaussian posterior densities.

\section*{Appendix}
A retraction on a manifold $M$ is a smooth map $R: TM \to M$ defined by $(x, X) \mapsto R_x(X)$ such that for any point $x \in M$, each curve $\gamma(t)=R_x(tX)$ satisfies $\gamma(0)=x$ and $\dot{\gamma}(0)=X$ \cite{boumal2023}. To construct a retraction for $\text{EM}(c)^{\frac{1}{2}} \coloneqq \left\{\sqrt{p_\theta}: p_\theta \in \text{EM}(c)\right\}$, we can use the construction of a retraction from the local coordinate as stated in Section 4.1.3 of \cite{absil2008} as follows:
\begin{align}
  R_{\sqrt{p_\theta}}:& T_{\sqrt{p_\theta}}\text{EM}(c)^{\frac{1}{2}} \to \text{EM}(c)^{\frac{1}{2}} \nonumber\\
  & V \mapsto \pi (\sqrt{p_\theta})(\varrho_\ast V).  \label{eq:retraction}
\end{align}

In \eqref{eq:retraction}, $\pi (\sqrt{p_\theta}) : \mathbb{R}^m \to \text{EM}(c)^{\frac{1}{2}}$ is defined as $\pi (\sqrt{p_\theta})(v) = \varrho^{-1}(v + \varrho(\sqrt{p_\theta}))$. In particular, using local vector representation with $v \in \mathbb{R}^m$ such that $V = \sum_{i=1}^m v^i \partial_i$, \eqref{eq:retraction} is equal to:
\begin{equation}
  R_{\sqrt{p_\theta}} (V) = \sqrt{p_{\theta + v}} \label{eq:Retraction_R_sqrt_p}
\end{equation}
Using this equation, it is straightforward to show that a curve $\gamma(t) \coloneqq R_{\sqrt{p_\theta}}(tV)$ satisfies $\gamma(0)=\sqrt{p_\theta}$ and $\dot{\gamma}(0)=V$. 

Observe that given $\theta \in \Theta$, we need to ensure that $t$ is selected from an open interval $I$ containing $0$ such that $\theta + tv \in \Theta$ for any $t \in I$. In the following proposition, we show that the existence of such an interval is guaranteed for the case of regular exponential families. 
\begin{prop}
  Let $\text{EM}(c)$ be a regular exponential family. For any $\theta \in \Theta \subseteq \mathbb{R}^m$ and $V = \sum_{i=1}^m v^i \partial_i$, there exists an open interval $I$ containing $0$ such that the curve $\gamma(t) \coloneqq R_{\sqrt{p_\theta}}(t V) \in \text{EM}(c)^{\frac{1}{2}}$ for all $t \in I$.
\end{prop}
\proof{
Consider the line $\ell = \left\{ \theta + t v : t \in (-\infty, \infty) \right\} \subset \mathbb{R}^m$. By \eqref{eq:Retraction_R_sqrt_p}, $\gamma(t)=\sqrt{p_{\theta+t v}}$. Since $\text{EM}(c)$ is a regular exponential family, $\Theta$ is an open subset of $\mathbb{R}^m$. Therefore, there exists an open convex neighborhood $U$ of $\theta \in \Theta$ such that the line section $\tilde{\ell} = U \cap \ell \subset \Theta$ is non-empty. Thus, the existence of the interval $I$ follows immediately. $\square$
}

\bibliographystyle{elsarticle-num} 
\bibliography{Zotero_BibTeX,additional_bibtex}
\end{document}